\documentclass[12pt, reqno]{amsart}
\usepackage{enumerate}
\usepackage{blkarray, bigstrut}
\usepackage{graphics, xcolor}
\usepackage{hyperref}
\hypersetup{colorlinks=true, linkcolor=blue, citecolor=blue}

 \newtheorem{thm}{Theorem}[section]
 \newtheorem{crlre}[thm]{Corollary}
 \newtheorem{lma}[thm]{Lemma}
 \newtheorem{ppsn}[thm]{Proposition}
 \theoremstyle{definition}
 \newtheorem{dfn}[thm]{Definition}
 \theoremstyle{remark}
 \newtheorem{rmrk}[thm]{Remark}
 
 \newtheorem{Problem}{Problem}
 \numberwithin{equation}{section}

 \newcommand{\D}{\mathbb{D}}
 \newcommand{\cir}{\mathbb{T}}
 \newcommand{\R}{\mathbb{R}}
 \newcommand{\C}{\mathbb{C}}
 \newcommand{\N}{\mathbb{N}}	
 \newcommand{\Z}{\mathbb{Z}}
 \newcommand{\hilh}{\mathcal{H}}
 \newcommand{\hilk}{\mathcal{K}}
 \newcommand{\bh}{\mathcal{B}(\hilh)}
 \newcommand{\sph}{\mathcal{S}_p(\hilh)}

 \newcommand{\lip}{\operatorname{Lip}}
 \newcommand{\lipp}{\operatorname{\mathcal{L}ip}}
 
 \newcommand{\Tr}{\operatorname{Tr}}

\begin{document}

\title[Lipschitz Estimates and an application to trace formulae]
 {Lipschitz Estimates and an application to trace formulae}

\author[Bhattacharyya]{Tirthankar Bhattacharyya}
\address{Department of Mathematics, Indian Institute of Science,	Bangalore 560012, India}
\email{tirtha@iisc.ac.in}

\author[Chattopadhyay]{Arup Chattopadhyay}

\address{Department of Mathematics, Indian Institute of Technology Guwahati, Guwahati, 781039, India}
\email {arupchatt@iitg.ac.in}

\author[Giri]{Saikat Giri}
\address{Department of Mathematics, Indian Institute of Technology Guwahati, Guwahati, 781039, India}
\email{saikat.giri@iitg.ac.in}

\author[Pradhan]{Chandan Pradhan}
\address{Department of Mathematics, Indian Institute of Science,	Bangalore 560012, India}
\email{chandan.pradhan2108@gmail.com}	



\date{\today}

\begin{abstract}
For each $p\in(1,\infty)$ and for every pair of contractions $(T_0,T_1)$ with $\|T_0\|<1$, we show that there exists a constant $d_{f, p,T_0}>0$ such that 
$$\|f(T_1)-f(T_0)\|_p\leq d_{f,p, T_0}\|T_1-T_0\|_p$$ 
for all Lipschitz functions $f$ on $\cir$. The functional calculus $f(T)$ for a contraction $T$ and for a Lipschitz function $f$ is formed through a unitary dilation of $T$ and is shown to be independent of the dilation being used. Using the estimate in the display above, we establish a modified Krein trace formula applicable to a specific category of pairs of contractions featuring Hilbert-Schmidt perturbations. An old result of Birman and Solomyak 
gives an expression of $f(U)-f(V)$ in the form of a double operator integral where $f$ is a Lipschitz function on the unit circle $\cir$ with $f'\in \lip_{\alpha}(\cir)$ for $\alpha>0$ and $(U,V)$ is a pair of unitary operators  with $U-V\in\mathcal{S}_{2}(\hilh)$ (the Hilbert-Schmidt class). We give an elementary proof of this formula for every Lipschitz function $f$ on the unit circle $\cir$, thereby enlarging the class of functions. As a consequence, we obtain the Schatten $2$-Lipschitz estimate $\|f(U)-f(V)\|_2\leq \|f\|_{\lip(\cir)}\|U-V\|_2$ for all Lipschitz functions $f:\cir\to\C$. 
\end{abstract}

\maketitle

\section{Introduction}

We denote the open unit disk in the complex plane $\C$ by $\D$ and the unit circle by $\cir$.

A remarkable result of Potapov and Sukochev \cite{PoSu11} shows that for a Lipschitz function $f$ on the real line $\R$ and a number  $p \in (1,\infty)$, there is a constant $c_p$ such that
\begin{align}\label{cp}
	\|f(A)-f(B)\|_p\leq c_p\,\|f\|_{\lip(\R)}\,\|A-B\|_p,
\end{align}
for any  self-adjoint (possibly unbounded) linear operators $A$ and $B$ such that $A-B$ is in the Schatten $p$-class $\sph$. Here $\| \cdot \|_p$ denotes the Schatten $p$-norm. For $p=\infty$, this is the operator norm and we shall denote that by $\| \cdot \|$. Also,  $\|f\|_{\lip(\R)}$ denotes the Lipschitz norm   of $f$:
\begin{equation}\label{lipnorm}
	\sup\left\{\frac{|f(x)-f(y)|}{|x-y|}:~x,y\in \R, x\neq y\right\}.
\end{equation}

The above result resolved a number of conjectures of Krein \cite{MGK}, Farforovskaya \cite{Farforovskaja74} and Peller \cite{Peller87}. For more on the history of this subject we refer the reader to the recent survey \cite{Peller16} by Aleksandrov and  Peller. In the unitary setting, an analogue of the above result was obtained in \cite[Theorem 2]{AyCoSu16}.

\noindent For each $p\in(1,\infty),$ there is a constant $d_p>0$ such that
\begin{align}\label{d_p}
	\|f(U_1)-f(U_0)\|_p\leq d_p\|f\|_{\lipp} \|U_1-U_0\|_p
\end{align}
where $$\|f\|_{\lipp}=\sup_{\substack{\lambda, \mu\in\R\\ \lambda\neq\mu}}\frac{|f(e^{i\lambda})-f(e^{i\mu})|}{|\lambda-\mu|},$$
for all Lipschitz functions $f:\mathbb{T}\to\mathbb{C}$ and all unitary operators $U_0$ and $U_1$ such that $U_1-U_0\in\mathcal S_p(\hilh)$. Further $d_p\leq 32(c_{p}+9)$, where $c_p$ is as in \eqref{cp}. 
And moreover,  $\|f\|_{\lipp} \leq \|f\|_{\lip(\cir)}$, where the Lipschitz norm  $\|f\|_{\lip(\cir)}$ of $f$ is defined similarly as in \eqref{lipnorm}. 

\noindent It was shown in {\normalfont\cite{CaMoPoSu14}} that $c_{p}=O(p^2/p-1)$, which is the optimal constant.

Motivated by \eqref{cp} and \eqref{d_p}, the main aim of this paper is to prove such a bound for contractions. 
As a consequence, this allows us to establish a trace formula for pairs of contractions like Krein's trace formula. 

To get into the specifics, an operator $T\in\bh$ is a contraction if $\|T\|\le 1$ and is a strict contraction if the inequality is strict. For a contraction $T$ on a Hilbert space $\hilh$, we shall always denote by $U_T$ {\em a unitary dilation} in the sense of Sz.-Nagy \cite{Sz.-Nagy},  see also \cite{Schaffer}, i.e., $U_T$ acts as a unitary operator on a Hilbert space $\hilk$ containing $\hilh$ as a closed subspace such that 

\begin{align*}
	T^{n}=P_{\hilh}U_{T}^{n}|_\hilh\hspace*{0.7cm}\text{~for~} n\in\mathbb{Z},
\end{align*}
where  $T^{-|n|}=(T^*)^{|n|}$, and $P_{\hilh}$ is the orthogonal projection from $\hilk$ onto $\hilh$.  
Let $N\in\N$. For a trigonometric polynomial $p_{N}(z)=\sum_{n=-N}^{N}a_{n}z^{n},~a_{n}\in\C$, define $p_{N}(T)$ as 
\begin{align*}
	&p_{N}(T)=\sum_{n=-N}^{-1}a_{n}T^{*|n|}+\sum_{n=0}^{N}a_{n}T^{n}.
\end{align*}
Hence, for any trigonometric polynomial $p$,
\begin{align*}
	&p(T)=P_{\hilh}p(U_{T})|_{\hilh}.
\end{align*}
By using continuous functional calculus, it is now natural that we define for any  $f\in C(\cir)$,
\begin{align*}
	&f(T)=P_{\hilh}f(U_T)|_{\hilh}.
\end{align*}
Clearly, $f(T)$ is a bounded operator satisfying {\em the von Neumann inequality}
\begin{align}
	\label{vN}\|f(T)\|\leq \|f\|_\infty.
\end{align}
Moreover, if $U_T$ and $V_T$ are two unitary dilations of $T$ on  Hilbert spaces $\hilk_1\supset\hilh$ and $\hilk_2\supset\hilh$ respectively, then for a trigonometric polynomial $p$, we have
\begin{align*}
	&\|P_{\hilh}f(U_T)|_{\hilh} -P_{\hilh}f(V_T)|_{\hilh} \| \\
	\le & \|P_{\hilh}f(U_T)|_{\hilh}-P_{\hilh}\,p(U_T)|_{\hilh}\|+\|P_{\hilh}f(V_T)|_{\hilh}-P_{\hilh}\,p(V_T)|_{\hilh}\|
\end{align*}
which by \eqref{vN} and the denseness of trigonometric polynomials in $C(\cir)$ proves that {\em $f(T)$ is independent of the choice of the unitary dilation $U_T$}. With this functional calculus (which is linear but not multiplicative) in hand, the main result of this note is as follows.

\begin{thm}\label{Thm3}
	Let $f:\cir\to\C$ be a Lipschitz function on $\cir$ and $1<p<\infty$. Then for any pair of contractions $(T_0,T_1)$ satisfying $\|T_{0}\|<1$ and the hypothesis $T_1-T_0\in\sph$, has the following estimate:
	\begin{align*}
		\|f(T_1)-f(T_0)\|_p\leq  ~2^{\frac{1}{p}}\left[\frac{2+(1-\|T_0\|^2)^{1/2}}{(1-\|T_0\|^2)^{1/2}}\right]~d_p \|f\|_{\lipp}\|T_1-T_0\|_p,
	\end{align*} 
	where $d_p$ and $ \|f\|_{\lipp}$ are as in \eqref{d_p}.
\end{thm}


Our main tool is the celebrated Sch\"affer dilation for contractions \cite{Schaffer} which perhaps first appeared in this context in \cite{Neid_1988} and \cite{Ry1}. It has been greatly used in \cite{KiSh07} and in \cite{MaNePe19}. In the literature, a complex-valued continuous function $f: I(\subseteq \R)\to \C$ is called {\em $\mathcal{S}_{p}$-Lipschitz} for $1\le p<\infty$ if there exists a constant $C_{f,\,p}$ such that 
\begin{equation*}
	\|f(A)-f(B)\|_p\leq C_{f,\,p}\|A-B\|_p,
\end{equation*}
for all self-adjoint operators $A, B$ with $\sigma(A)\cup\sigma(B)\subset I$ and $A-B\in\sph$. The $\mathcal{S}_p$-Lipschitz property for pairs of contractions has been studied in  \cite[Theorem 4.2]{KiSh07} and \cite[Theorem 6.4]{KiPoShSu}. The existing results are for Lipschitz functions in the {\em disk algebra} $\mathcal{A}(\mathbb T)$ of functions which are holomorphic in $\D$ and continuous in $\overline{\D}$. We extend this to arbitrary Lipschitz functions on the unit circle with the cost of one of the contractions being strict and the constant being dependent on the strict contraction. 

The  article is structured as follows.	Section \ref{TheMainResult} deals with Lipschitz estimates for a pair of contractions when one of them is strict.  In Section \ref{Sec3}, we establish the modified Krein trace formula. In Section \ref{DOI}, we will give a brief discussion on the double operator integral. En route, we establish the operator Lipschitz estimate for pairs of unitaries via the double operator integral technique. We end this note by proposing an open problem in Section \ref{Sec4}.

\section{The Main Result} \label{TheMainResult}

As is well-known, there is a trade off between the conditions on the operators and the conditions on the functions when proving perturbation estimates. The following result has fewer conditions on the operators whereas Theorem \ref{Thm3} has fewer conditions on the functions involved.
\begin{ppsn}\label{ppsn1}
	Let $T_0$ and $T_1$ be two contractions on $\hilh$ such that $T_1-T_0\in\sph$, $1\leq p\le\infty$. Let $f$ be a function on the unit circle $\cir$ such that $\sum\limits_{n\in\Z}|n\hat{f}(n)|<\infty$. Then
	\begin{align*}
		\|f(T_1)-f(T_0)\|_p\leq~\left(\sum\limits_{n\in\Z\setminus \{0\}}|n\hat{f}(n)|\right) \|T_1-T_0\|_p.
	\end{align*}
	
\end{ppsn}
\begin{proof}
	The conclusion follows from the equality
	\begin{align*}
		&f(T_1)-f(T_0)\\
		=&\sum\limits_{n=1}^{\infty}\hat{f}(-n)\left({T_1^*}^n-{T_0^*}^n\right)+\sum\limits_{n=1}^{\infty}\hat{f}(n)\left(T_1^n-T_0^n\right)\\
		=&\sum\limits_{n=1}^{\infty}\hat{f}(-n)\sum\limits_{k=0}^{n-1}{T_1^*}^k\left({T_1^*}-{T_0^*}\right){T_0^*}^{n-k-1}+\sum\limits_{n=1}^{\infty}\hat{f}(n)\sum\limits_{k=0}^{n-1}{T_1}^k\left({T_1}-{T_0}\right){T_0}^{n-k-1}.\\
	\end{align*}
	
\end{proof}

In our journey towards getting results with fewer conditions on the function, we first obtain a preliminary result by restricting the class of operators.
\begin{lma} If $f\in\lip(\cir)$ satisfies $\sum\limits_{n\in\Z}|n\hat{f}(n)|<\infty$, and if both $T_0$ and $T_1$ are strict contractions such that $T_1-T_0\in \mathcal{S}_p(\hilh)$ for $1\leq p\le\infty$, then 
	\begin{align*}
		\|f(T_1)-f(T_0)\|_p\leq~\sqrt{2}\,\|f\|_{\lip(\cir)}\,(1- \max\{\|T_1\|, \|T_0\|\}^2)^{-1/2}\|T_1-T_0\|_p.
	\end{align*}
\end{lma}
\begin{proof}
	Suppose $0\leq \|T_0\|\leq\|T_1\|<1$.  Using the same equality as in the proof of the last lemma,
	\begin{align*}
		\|f(T_1)-f(T_0)\|_p\leq 
		& \left(\sum\limits_{n\in\Z\backslash\{0\}}|n\hat{f}(n)|\|T_1\|^{|n|-1}\right)\|T_1-T_0\|_p\\
		\leq&~\sqrt{2}\left[\sum\limits_{n\in\Z\backslash\{0\}}|n\hat{f}(n)|^2\right]^{1/2}\left[\sum\limits_{n=0}^{\infty}\|T_1\|^{2n}\right]^{1/2}\|T_1-T_0\|_p\\
		=&~\sqrt{2}\,\|f'\|_2\,\left(1-\|T_1\|^2\right)^{-1/2}\|T_1-T_0\|_p\\
		\le&~\sqrt{2}\,\|f\|_{\lip(\cir)}\,\left(1-\|T_1\|^2\right)^{-1/2}\|T_1-T_0\|_p.
	\end{align*}
\end{proof}

\begin{lma}\label{Lem2}
	Let $A$ and $B$ be two positive contractions on $\mathcal{H}$. If $B\geq \delta I> 0$ for some positive $\delta$ and if for any $p\in[1,\infty]$, $\left(A^2-B^2\right)\in \mathcal{S}_p(\hilh)$, then $\left(A-B\right)\in\mathcal{S}_p(\hilh)$, and 
	\begin{align*}
		\|A-B\|_p\leq \dfrac{1}{\delta}\|A^2-B^2\|_p.
	\end{align*}
\end{lma}
\begin{proof}
	Lemma 2.1 in \cite{ChSi21} tells us that we have the following identity in  $\mathcal{B}(\mathcal{H})$: 
	\begin{equation}\label{eq12}
		A-B =  \int_0^{+\infty} \textup{e}^{-tA}~\big(A^2-B^2\big) ~\textup{e}^{-tB} ~dt,
	\end{equation}  
	where the right hand side  exists as an improper strong Riemann-Bochner integral. The hypothesis in the statement implies $B$ has a trivial kernel and that is why we can use Lemma 2.1 of \cite{ChSi21}. Moreover, the integral in \eqref{eq12} converges in $\sph$-topology since $\left\| \textup{e}^{-tA}\right\|\leq 1$, $\left\| \textup{e}^{-tB}\right\|\leq  \textup{e}^{-\delta t}$. Hence 
	\begin{align*}
		\|A-B\|_p\leq\int_{0}^{+\infty}\|A^2-B^2\|_p\,e^{-t\delta}dt= \dfrac{1}{\delta}\|A^2-B^2\|_p.
	\end{align*}
	This completes the proof of the lemma.    
\end{proof}	 

Since $T$ is a contraction, $I - T^*T$ and $I - TT^*$ are positive operators, and hence we can take their positive square roots: {\em the defect operators} $D_{T}=(I-T^{*}T)^{1/2}$ and $D_{T^{*}}=(I-TT^{*})^{1/2}$. The following corollary is an immediate consequence of Lemma \ref{Lem2}.
\begin{crlre}\label{Crl1}
	Let $T_{0}$ and $T_{1}$ be two contractions on $\hilh$ such that one of them is strict. If $T_{1}-T_{0}\in\sph$, then $D_{T_{1}}-D_{T_{0}}$ and $D_{T_{1}^{*}}-D_{T_{0}^{*}}$ both are in $\sph$.
\end{crlre}

Sch\"affer \cite{Schaffer} found an explicit power unitary dilation of a contraction $T$ on a Hilbert space $\hilh$. Let $\ell_2(\hilh)=\hilh\oplus\hilh\oplus\cdots$. The dilation space is $\ell_2(\hilh)\oplus \hilh\oplus \ell_2(\hilh)$ and the unitary operator is 
\begin{align}\label{newschafferdil}
	U_{T}=\begin{blockarray}{ccccc}
		\ell_2(\hilh) & \hilh & \ell_2(\hilh) &  \\[4pt]
		\begin{block}{[ccc]cc}
			S^*& 0& 0 & \ell_2(\hilh) \\[3pt]
			D_{T^*}P_\hilh& T& 0 & \hilh\\[3pt]
			-T^*P_\hilh& D_{T}& S& \ell_2(\hilh)\\[3pt]
		\end{block}
	\end{blockarray},
\end{align}
where $S$ is the unilateral shift on $\ell_2(\hilh)$ of multiplicity $\dim(\hilh)$ and $P_\hilh$ is the orthogonal projection from $\ell_2(\hilh)$ onto $\hilh\oplus 0\oplus0\oplus\cdots$.\\

\begin{thm}\label{Thm2}
	Let $1\leq p<\infty$ and $T_0,T_1$ be two contractions on $\hilh$ having the following properties:
	\begin{enumerate}[{\normalfont(i)}]
		\item\label{Case1} $T_0$ is strict.
		\item\label{Case2} $T_1-T_0\in\sph$.
	\end{enumerate}Let $U_{T_0}$ and $U_{T_1}$ on $\ell_2(\hilh)\oplus\mathcal{H}\oplus\ell_2(\hilh)$ be Sch\"affer matrix dilations for $T_0$ and $T_1$ respectively. Then  $$\|U_{T_1}-U_{T_0}\|_p\le~2^{\frac{1}{p}}\left[\frac{2+(1-\|T_0\|^2)^{1/2}}{(1-\|T_0\|^2)^{1/2}}\right]\|T_1-T_0\|_p.$$ 
\end{thm}
\begin{proof}
	The difference of $U_{T_0}$ and $U_{T_1}$ is given by
	\begin{align}\label{eq13}
		\nonumber &U_{T_1}-U_{T_0}\\
		\nonumber=&
		\begin{bmatrix}
			0&0&0\\
			(D_{T_1^*}-D_{T_0^*})P_\hilh&T_1-T_0&0\\
			(T_0^*-T_1^*)P_\hilh&D_{T_1}-D_{T_0}&0
		\end{bmatrix} \\
		=&\begin{bmatrix}
			0&0&0\\
			0&T_1-T_0&0\\
			(T_0^*-T_1^*)P_\hilh&0&0
		\end{bmatrix} +\begin{bmatrix}
			0&0&0\\
			(D_{T_1^*}-D_{T_0^*})P_\hilh&0&0\\
			0&D_{T_1}-D_{T_0}&0
		\end{bmatrix} .
	\end{align}
	From the identity \eqref{eq13} and the Corollary \ref{Crl1}, we arrive at
	\begin{align}
		\label{eq14}\|U_{T_1}-U_{T_0}\|_p\leq 2^{\frac{1}{p}}~\|T_1-T_0\|_p+\Big[\|D_{T_1^*}-D_{T_0^*}\|_p^p+\|D_{T_1}-D_{T_0}\|_p^p\Big]^{\frac{1}{p}}.
	\end{align}
	Since $T_0$ is a strict contraction, the spectrum $\sigma(T_0^*T_0)$ of $T_0^*T_0$  is contained in $[0,\|T_0^*T_0\|]=\left[0,\|T_0\|^2\right]\subseteq[0,1)$. Let $\delta=\left(1-\|T_0\|^2\right)^{1/2}>0$. By spectral theorem, we have $D_{T_0}\geq \delta I>0$. Applying Lemma \ref{Lem2} for the pair $(A=D_{T_1},B=D_{T_0})$ we obtain
	\begin{align}
		\label{eq15}\|D_{T_1}-D_{T_0}\|_p\leq \frac{2}{\left(1-\|T_0\|^2\right)^{1/2}}\|T_1-T_0\|_p.
	\end{align}
	By a similar argument we also have
	\begin{align}
		\label{eq16}\|D_{T^*_1}-D_{T^*_0}\|_p\leq \frac{2}{\left(1-\|T_0\|^2\right)^{1/2}}\|T_1-T_0\|_p.
	\end{align}
	Therefore, using \eqref{eq15} and \eqref{eq16} in \eqref{eq14} we derive
	\[\|U_{T_1}-U_{T_0}\|_p\le ~2^{\frac{1}{p}}\left[\frac{2+(1-\|T_0\|^2)^{1/2}}{(1-\|T_0\|^2)^{1/2}}\right]\|T_1-T_0\|_p.\]
	This completes the proof.
\end{proof}

The above implies the main result.

\begin{thm}\label{Thm3}
	Let $f:\cir\to\C$ be a Lipschitz function on $\cir$ and $1<p<\infty$. Then for any pair of contractions $(T_0,T_1)$ satisfying $\|T_{0}\|<1$ and the hypothesis $T_1-T_0\in\sph$, one has the following estimate:
	\begin{align*}
		\|f(T_1)-f(T_0)\|_p\leq  ~2^{\frac{1}{p}}\left[\frac{2+(1-\|T_0\|^2)^{1/2}}{(1-\|T_0\|^2)^{1/2}}\right]~d_p \|f\|_{\lipp}\|T_1-T_0\|_p,
	\end{align*} 
	where $d_p$ and $ \|f\|_{\lipp}$ are as in \eqref{d_p}.
\end{thm}
\begin{proof}
	Let  $U_{T_0}$ and $U_{T_1}$ be the unitary dilations for $T_0$ and $T_1$ respectively in the Sch\"affer form. Since $f$ is a Lipschitz function on $\cir$, by Theorem \ref{Thm2}, we have 
	\begin{align*}
		\|f(T_1)-f(T_0)\|_p&\le \|f(U_{T_{1}})-f(U_{T_{0}})\|_p\\
		&\le ~2^{\frac{1}{p}}\left[\frac{2+(1-\|T_0\|^2)^{1/2}}{(1-\|T_0\|^2)^{1/2}}\right]~d_p \|f\|_{\lipp}\|T_1-T_0\|_p.
	\end{align*}
	This completes the proof.	
\end{proof}

The technique of using a unitary dilation allow us to extend the function class considered in \cite[Corollary 5.4.]{KiSh07} from $\mathcal{A}(\cir) \cap \lip(\cir)$ to $\lip(\cir)$.

\begin{thm}\label{Thm_kis}
	Let $f:\cir\to\C$ be a Lipschitz function on $\cir$ and $1< p<\infty$. Then for any pair of contractions $(T_0,T_1)$ satisfying $T_1-T_0\in\mathcal{S}_{p/2}(\hilh)$, we have the following estimate:
	\begin{align}\label{lipestnew}
		\|f(T_1)-f(T_0)\|_p\leq  k_p\, d_p\,  \|f\|_{\lipp}\max\{\|T_1-T_0\|_p,\, \|T_1-T_0\|_{p/2}^{1/2}\},
	\end{align}  
	where $d_p$, $\|f\|_{\lipp}$ are as in \eqref{d_p} and 
	\[
	k_p=2^{1/p}(1+\sqrt{2}) \text{ if } p\geq 2, \quad \text{and} \quad k_p= 2^{1/p}(1+2^{2/p}) \text{ if } p< 2.
	\]
\end{thm}

\begin{proof}
	Let $U_{T_0},U_{T_1}$ be two Sch\"affer unitary dilations for $T_0$ and $T_1$ respectively. By \eqref{d_p}, we have
	\begin{align*}
		\|f(T_1)-f(T_0)\|_p &\le \|f(U_{T_{1}})-f(U_{T_{0}})\|_p \\
		&\le~d_p \|f\|_{\lipp}\|U_{T_1}-U_{T_0}\|_p.
	\end{align*}
	Applying Proposition 5.2 (ii) of \cite{KiSh07} on the right-hand side of the above inequality, we obtain \eqref{lipestnew}.
\end{proof}

\section{Application: A modified Krein trace formula}\label{Sec3}
In this section, we apply Theorem \ref{Thm3} to prove the existence of the spectral shift function and the associated modified Krein trace formula for a pair of contractions differing by a Hilbert-Schmidt operator.

The spectral shift function introduced by Lifshitz in \cite{Li52} and generalized  by Krein guarantees that for an arbitrary pair of self-adjoint operators $A$ and $B$ with $A-B$ in the trace class ideal, there exists a unique integrable function $\xi$ on $\Omega=\R$  such that the trace formula
\begin{align}
	\label{Krein}&\Tr\left(f(A)-f(B)\right)=\int_{\Omega}f'(x)\,\xi(x)\,dx
\end{align}
holds if the Fourier transform of $f'$ is in $L^{1}(\R)$ (see \cite{Kr53}). The function $\xi$ is called the {\it spectral shift function} and the formula \eqref{Krein} is known as the {\it Krein trace formula}.

A similar result was obtained in \cite{Kr62} for pairs of unitary operators $(U, V)$ with $U-V \in\mathcal{S}_{1}(\hilh)$ and $\Omega=\cir$. The class of functions there needed to satisfy that $f'$ has an absolutely convergent Fourier series. More than five decades later, a better sufficient condition on $f$ was found in \cite[Theorem 4.1]{AlPe16}. They showed that \eqref{Krein} holds for any {\em operator Lipschitz function} $f$ on $\cir$, i.e., for those $f$ for which there is a constant $C_f$ such that 
\begin{equation*}
	\|f(U)-f(V)\|\le C_f \|U-V\|
\end{equation*}
for all unitary operators $U$ and $V$.  

The existence of a complex-valued integrable spectral shift function $\xi$ for a pair of contractions $(T_0, T_1)$ with trace class difference $T_1-T_0\in\mathcal{S}_1(\hilh)$ corresponding to all $f\in \mathcal{A}(\cir)$, which are operator Lipschitz on $\cir$, was proved in \cite{MN2015} (under an additional assumption $\rho(T_0)\cap \mathbb{D}\neq\emptyset$, where $\rho(T_0)$ is the resolvent set of $T_0$) and in \cite{MaNePe19} (in full generality). For more on the Krein trace formula, we refer to \cite{K1987,Lan1965,Ry1,Ry2,Ry3,Ry4}.

The following theorem, a higher-order version of which appeared in \cite[Theorem 4.5]{Sk17}, will be useful.
\begin{thm}\label{Thm8}
	Let $U$ and $V$ be two unitary operators acting on $\hilh$ and let
	\begin{align*}
		f\in\mathcal{F}_{1}(\cir):=\left\{f\in C^{1}(\cir):~\sum_{n\in\Z}|n\widehat{f}(n)|<\infty\right\}.
	\end{align*}	
	Let $U-V\in\mathcal{S}_{2}(\hilh)$. Then for all $Z\in\mathcal{S}_2(\hilh)$, there exists a function $\eta=\eta_{\,U,\,V,\,Z}\in L^{1}([0, 2\pi])$ satisfying
	\begin{align*}
		&\Tr\Big[(f(U)-f(V))Z\Big]=\int_{0}^{2\pi}f'(e^{it})\,\eta(t)\,dt,
	\end{align*}
	for every $f\in \mathcal{F}_{1}(\cir)$.
\end{thm}

We omit the proof because the arguments are similar to \cite[Theorem 4.5]{Sk17}. The main result in this section is the following theorem.
\begin{thm}\label{Thm7}
	Let $T_{0}$ and $T_{1}$ be two contractions and $U_{T_0}$, $U_{T_1}$ be the unitary dilations of $T_{0}$ and $T_{1}$, given by \eqref{newschafferdil} respectively. If $U_{T_1}-U_{T_0}\in\mathcal{S}_{2}(\hilh)$, then for each $f\in\lip(\cir)$,
	\begin{align}\label{newtrcl}
		\Big[\left(f(T_1)-f(T_0)\right)X\Big]\in \mathcal{S}_1(\mathcal{H}),
	\end{align}
	for every $X\in\mathcal{S}_2(\hilh)$.
	Moreover, for each $X\in\mathcal{S}_2(\hilh)$, there exists a function $\eta=\eta_{\,T_{0},\,T_{1},\,X}\in L^{1}([0,2\pi])$ (unique upto an additive constant) satisfying
	\begin{align}
		\label{Thm7eq1}&\Tr\Big[(f(T_{1})-f(T_{0}))X\Big]=\int_{0}^{2\pi}\frac{d}{dt}f(e^{it})\,\eta(t)\,dt,
	\end{align} for every $f\in C^1(\cir)$.
\end{thm}
\begin{proof}
	By H\"older's inequality and \eqref{d_p},
	\begin{align}
		\nonumber\left\|\Big[(f(T_{1})-f(T_{0}))X\Big]\right\|_1&\le\|f(T_{1})-f(T_{0})\|_{2}~\|X\|_{2}\\
		\nonumber&\le\|f(U_{T_{1}})-f(U_{T_{0}})\|_{2}~\|X\|_{2}\\
		\label{newThm7eq2}&\le d_2 \|f\|_{\lip(\cir)}\,\|U_{T_{1}} - U_{T_{0}}\|_{2}\, \|X\|_{2}.
	\end{align}
	Therefore \eqref{newtrcl} holds true. Now for $f\in C^1(\cir)$, from \eqref{newThm7eq2}, we have
	\begin{align*}
		&\left|\Tr\Big[(f(T_{1})-f(T_{0}))X\Big]\right|\leq d_2\|f'\|_{\infty}\,\|U_{T_1}-U_{T_0}\|_2\,\|X\|_{2}.
	\end{align*}
	Hence there exists a complex Borel measure $\mu$ on $\cir$ such that 
	\begin{align}
		\label{Thm7eq3}&\Tr\Big[(f(T_{1})-f(T_{0}))X\Big]=\int_{\cir}f'(z)\,d\mu(z),
	\end{align}
	for all $f\in C^{1}(\cir)$. Now by change of variable in the integration of \eqref{Thm7eq3} we arrive at
	\begin{align}
		\label{Thm7eq5}&\int_{\cir}f'(z)\,d\mu(z)=\int_{0}^{2\pi}\frac{d}{dt}f(e^{it})\,d\nu(t),
	\end{align}
	where $\nu$ is the push forward measure for $\mu$ on $[0, 2\pi]$. On the other hand, for $f\in\mathcal{F}_{1}(\cir)\subset C^{1}(\cir)$, Theorem \ref{Thm8}  and the Sch\"affer matrix representations of $U_{T_1}$ and $U_{T_0}$ and the cyclicity property of trace together give
	\begin{align}
		\nonumber\Tr\Big[(f(T_{1})-f(T_{0}))X\Big]&=\nonumber\Tr\Big[{\hat{P}}_\hilh\,\left(f(U_{T_{1}})-f(U_{T_{0}})\right)\,{\hat{P}}_\hilh\,\left({\hat{P}}_\hilh\, X\,{\hat{P}}_\hilh\right)\Big]\\
		\nonumber&=\Tr\Big[\left(f(U_{T_{1}})-f(U_{T_{0}})\right)\left({\hat{P}}_\hilh\, X\,{\hat{P}}_\hilh\right)\Big]\\
		\label{Thm7eq4}&=\int_{0}^{2\pi}\frac{d}{dt}f(e^{it})\,\eta(t)\,dt,
	\end{align}
	where ${\hat{P}}_\hilh$ is the orthogonal projection from $\ell_2(\hilh)\oplus\hilh\oplus\ell_2(\hilh)$ onto $0\oplus\hilh\oplus0$, and $\eta=\eta_{\,U_{T_{0}},\,U_{T_{1}},\,{\hat{P}}_\hilh\, X\,{\hat{P}}_\hilh}\in L^{1}[0,2\pi]$. Therefore, from \eqref{Thm7eq3}, \eqref{Thm7eq5} and \eqref{Thm7eq4} we arrive at
	\begin{align}
		\nonumber&\int_{0}^{2\pi}\frac{d}{dt}f(e^{it})\,\eta(t)dt=\int_{0}^{2\pi}\frac{d}{dt}f(e^{it})\,d\nu(t),\hspace*{1cm}\text{for all }~f\in\mathcal{F}_{1}(\cir),
	\end{align}
	which implies,
	\begin{align*}
		&\int_{0}^{2\pi}e^{int}\,\eta(t)\,dt=\int_{0}^{2\pi}e^{int}\,d\nu(t),\hspace*{1cm}\text{for all }~n\in\Z\backslash\{0\},
	\end{align*}
	which further implies that $d\nu(t)=(\eta(t)+c)dt$, for some constant $c$. This establishes \eqref{Thm7eq1}. Let $\eta_{1}$ and $\eta_{2}$ be two $L^{1}[0, 2\pi]$ functions satisfying \eqref{Thm7eq1}, then 
	\begin{align*}
		&\int_{0}^{2\pi}e^{int}\,\eta_{1}(t)\,dt=\int_{0}^{2\pi}e^{int}\,\eta_{2}(t)\,dt,\hspace*{1cm}\text{for all }~n\in\Z\backslash\{0\}.
	\end{align*}
	This shows that $\eta_{1}$ and $\eta_{2}$ are unique upto an additive constant. This completes the proof of the theorem.
\end{proof}
\begin{crlre}
	Thanks to Theorems \ref{Thm2} and \ref{Thm3}, the above Theorem \ref{Thm7} holds for all pairs of contractions, with one of them being strict, when their difference is a Hilbert-Schmidt operator.
\end{crlre}

\section{The Hilbert-Schmidt case}\label{DOI}

The main theorem of this section is the following.
\begin{thm}\label{Thm6}
	Let $U$ and $V$ be two unitary operators on $\hilh$ with spectral measures $E$ and $F$ respectively. Suppose $U-V\in{\mathcal S}_{2}(\hilh)$. Then
	\begin{align}\label{doiuniformula}
		&f(U)-f(V)=\int_{0}^{2\pi}\int_{0}^{2\pi}\frac{f(e^{i\lambda})-f(e^{i\mu})}{e^{i\lambda}-e^{i\mu}}\,dE(\lambda)\,(U-V)\,dF(\mu),
	\end{align}
	for any Lipschitz function $f:\cir\to\C$. Furthermore,
	\begin{align}
		\label{eq27}\|f(U)-f(V)\|_{2}\le\|f\|_{\lip(\cir)}\|U-V\|_{2}.
	\end{align}
\end{thm}

This is a generalization of a result of Birman and Solomyak \cite{BiSol66}. They considered $f^\prime$ in the class $\lip_\alpha(\mathbb T)$ defined below and used a different method of proof.
\begin{dfn}
	For $\alpha>0$, a function $f:\mathbb T\,\to\C$ is said to be $\alpha$-Lipschitz if there is a constant $C>0$ such that
	\begin{align*}
		|f(z)-f(w)|\le C|z-w|^{\alpha},
	\end{align*}
	for arbitrary $z, w\in\mathbb T$ (we use the notation $\lip_\alpha(\mathbb T)$ for the class of $\alpha$-Lipschitz functions on $\mathbb T$, and $\lip(\mathbb T)$ for $\lip_1(\mathbb T)$). The Lipschitz norm $\|f\|_{\lip(\mathbb T)}$ of $f$ is the smallest positive value of $C$.
\end{dfn}

A brief discussion of the double operator integral following Birman and Solomyak in \cite{BiSol66} (see also \cite{BiSol67}, \cite{BiSol73}) is necessary. Given two spectral measures $E(\cdot)$ and $F(\cdot)$ on the Borel $\sigma$-algebra $\mathcal{B}([0,2\pi])$, consider the spectral measure defined on the product sets of the $\sigma$-algebra $\mathcal{B}([0,2\pi])\times\mathcal{B}([0,2\pi]) $ by  
\begin{align}\label{eq25}
	G(\delta\times\sigma)(X):=E(\delta)XF(\sigma),\hspace*{0.7cm}\forall~ X\in\mathcal{S}_{2}(\hilh)
\end{align}
and extended on $\mathcal{B}([0,2\pi]\times[0,2\pi]) $ using standard techniques. 
\begin{dfn}
	Let $\Phi\in L^{\infty}([0, 2\pi]\times[0, 2\pi], G)$, where $G$ is defined as in \eqref{eq25}. The Birman-Solomyak double operator integral \,$T_{\Phi}^{G}:\mathcal{S}_{2}(\hilh)\to\mathcal{S}_{2}(\hilh)$ is defined as the integral of the symbol $\Phi$ with respect to the spectral measure $G(\cdot)$, that is,
	\begin{align}
		\nonumber T_{\Phi}^{G}(X):&=\int_{0}^{2\pi}\int_{0}^{2\pi}\Phi(\lambda_{1},\lambda_{2})\,dG(\lambda_{1},\lambda_{2})(X),\hspace*{0.5cm}X\in\mathcal{S}_{2}(\hilh).
	\end{align}
\end{dfn}
The notation 
\begin{align}
	\label{eq17}&T_{\Phi}^{G}(X):=\int_{0}^{2\pi}\int_{0}^{2\pi}\Phi(\lambda_{1},\lambda_{2})\,dE(\lambda_{1})\,X\,dF(\lambda_{2}),\hspace*{0.5cm}X\in\mathcal{S}_{2}(\hilh),
\end{align}
will be used frequently.  We note that $T_{\Phi}^{G}$ is bounded on $\mathcal{S}_{2}(\hilh)$ and $\|T_{\Phi}^{G}\|=\|\Phi\|_{\infty}$.

The following properties of $T_{\Phi}^{G}$ on $\mathcal{S}_{2}(\hilh)$ are useful:
\begin{enumerate}[(a)]
	\item\label{Prop1} $T_{\alpha\Phi+\beta\Psi}=\alpha T_{\Phi}+\beta T_{\Psi},\hspace*{0.3cm}\alpha, \beta\in\C$;
	\item\label{Prop2}$T_{\Phi\Psi}=T_{\Phi}T_{\Psi};$
	\item\label{Prop3}If $\Phi(\lambda_{1},\lambda_{2})=\phi_{1}(\lambda_{1})\cdot \phi_{2}(\lambda_{2})$ with  $\phi_1, \phi_2\in L^{\infty}([0, 2\pi])$, then 
	\begin{align*}
		T_{\Phi}^{G}(X):&=\left(\int_{0}^{2\pi}\phi_{1}(\lambda_{1})dE(\lambda_{1})\right)\cdot X\cdot\left(\int_{0}^{2\pi}\phi_{2}(\lambda_{2})\,dF(\lambda_{2})\right)
	\end{align*}
	holds for all $X\in\mathcal{S}_{2}(\hilh)$.
\end{enumerate}

\begin{rmrk}
	The Birman-Solomyak result can be stated using the function 
	\begin{align}
		\label{eq21}h^{[1]}(\lambda, \mu)=\begin{cases}\frac{f(e^{i\lambda})-f(e^{i\mu})}{e^{i\lambda}-e^{i\mu}}&\text{if $\lambda\neq\mu$},\\
			0& \text{if $\lambda=\mu$}.\\
		\end{cases}
	\end{align} 
	Taking $E$ and $F$ in equation \eqref{eq17} to be the spectral measures of $U$ and $V$ respectively, we know from \cite[Theorem 11]{BiSol66} that $T_{h^{[1]}}^G(U-V)$ is a Hilbert-Schmidt operator and, in fact, equals $f(U)-f(V)$ for $f'\in \lip_{\alpha}(\cir)$ for $\alpha>0$. 
\end{rmrk}
\noindent We give an elementary proof of Theorem \ref{Thm6} using just the Lipschitz property of $f$.
\begin{lma}\label{Thm4}
	Let $f:\cir\to\C$ be a Lipschitz function on $\cir$. Then for any pair of unitaries $U$, $V$ and for $X\in\mathcal{S}_{2}(\hilh)$, we have 
	\[f(U)X-Xf(V)=T_{h^{[1]}}^{G}\left(UX-XV\right), \text{ and }\]
	\begin{equation*}
		\|f(U)X-Xf(V)\|_{2}\le\|f\|_{\lip(\cir)}\|UX-XV\|_{2}.
	\end{equation*}
\end{lma}

\begin{proof}
	By the spectral theorem we have
	\begin{align*}
		U=\int_{0}^{2\pi}e^{i\lambda}dE(\lambda)\hspace*{0.5cm}\text{and}\hspace*{0.5cm}V=\int_{0}^{2\pi}e^{i\mu}dF(\mu),
	\end{align*}
	where $E(\cdot)$ and $F(\cdot)$ are the spectral measures associated with the unitaries $U$ and $V$ respectively.
	Let $p_{1}(\lambda,\mu)=e^{i\lambda}$ and $p_{2}(\lambda,\mu)=e^{i\mu}$. Therefore  $T_{p_{1}}^{G}(X)=UX$, $T_{p_{2}}^{G}(X)=XV$ and $T_{f\circ p_{1}}^{G}(X)=f(U)X$, $T_{f\circ p_{2}}^{G}(X)=Xf(V)$, which are the immediate implications of \eqref{Prop3}.
	
	Again \eqref{Prop1} and \eqref{Prop2} imply
	\begin{align}
		\nonumber f(U)X-Xf(V)&=T_{f\circ p_{1}}^{G}(X)-T_{f\circ p_{2}}^{G}(X)\\
		\nonumber&=T_{(f\circ p_{1}-f\circ p_{2})}^{G}(X)\\
		\nonumber&=T_{(p_{1}-p_{2})h^{[1]}}^{G}(X)\\
		&=T_{h^{[1]}}^{G}\left(T_{p_{1}}^{G}(X)-T_{p_{2}}^{G}(X)\right)  =T_{h^{[1]}}^{G}\left(UX-XV\right), \label{eq20}
	\end{align}
	where $h^{[1]}(\cdot,\cdot)$ is as given in \eqref{eq21}. Therefore, \eqref{eq20} further implies that
	\begin{align*}
		\|f(U)X-Xf(V)\|_{2}\le\|h^{[1]}\|_{\infty}\|UX-XV\|_{2}.
	\end{align*}
	This completes the proof since $\|h^{[1]}\|_{\infty}=\|f\|_{\lip(\cir)}.$
\end{proof}

We shall need the following theorem which is a special case of \cite[Theorem 2.2]{AcSi14}. The original idea goes back to Voiculescu \cite{Voiapp}.
\begin{thm}\label{Thm5}
	Let $T$ be a normal contraction on $\hilh$ and $V\in{\mathcal S}_{2}(\hilh)$. Then there exists a sequence $\{P_{n}\}$ of finite rank projections, such that $\{P_{n}\}\uparrow I$ and each of the following terms 
	\begin{align*}
		(i)~\|P_{n}^{\perp}T P_{n}\|_{2},\hspace*{0.5cm}(ii)~\|P_{n}T P_{n}^{\perp}\|_{2},\hspace*{0.5cm}(iii)~\|P_{n}^{\perp}V\|_{2},\hspace*{0.5cm}(iv)~\|P_{n}^{\perp}V^{*}\|_{2}
	\end{align*}
	converge to $0$ as $n\to\infty$.
\end{thm}

{\em Proof of Theorem \ref{Thm6}}: 
Applying Theorem \ref{Thm4} to $\left(f(U)P_{n}-P_{n}f(V)\right)$, where $P_n$ are as in Theorem \ref{Thm5}, we get
\begin{align}
	\label{eq22}\|f(U)P_{n}-P_{n}f(V)\|_{2}\le\|f\|_{\lip(\cir)}\|UP_{n}-P_{n}V\|_{2}.
\end{align}
Again, note that
\begin{align}
	\nonumber\|(UP_{n}-P_{n}V)-(U-V)\|_{2}&=\|UP_{n}^{\perp}-P_{n}^{\perp}V\|_{2}\\
	\nonumber&=\|P_{n}UP_{n}^{\perp}+P_{n}^{\perp}UP_{n}^{\perp}-P_{n}^{\perp}VP_{n}-P_{n}^{\perp}VP_{n}^{\perp}\|_{2}\\
	\label{eq26}&\le\|P_{n}UP_{n}^{\perp}\|_{2}+\|P_{n}^{\perp}VP_{n}\|_{2}+\|P_{n}^{\perp}(U-V)\|_{2}.
\end{align}
Therefore, using Theorem \ref{Thm5} and \eqref{eq26}, we have $\|(UP_{n}-P_{n}V)-(U-V)\|_{2} \to 0$ as $n\to\infty$. As a result,
\begin{align}
	\nonumber&\lim_{n\to\infty}\|T_{h^{[1]}}^{G}(UP_{n}-P_{n}V)-T_{h^{[1]}}^{G}(U-V)\|_{2}=0 \; \text{ and }\\
	\nonumber&\lim_{n\to\infty}\|(f(U)P_{n}-P_{n}f(V))-T_{h^{[1]}}^{G}(U-V)\|_2=0,
\end{align}
where $h^{[1]}$ and $G(\cdot)$ are the same as in \eqref{eq21} and \eqref{eq25} respectively. Hence
\begin{align*}
	f(U)-f(V)&=T^{G}_{h^{[1]}}(U-V)\\
	&=\int_{0}^{2\pi}\int_{0}^{2\pi}\frac{f(e^{i\lambda})-f(e^{i\mu})}{e^{i\lambda}-e^{i\mu}}\,dE(\lambda)\,(U-V)\,dF(\mu),
\end{align*}
which implies
\begin{align*}
	\|f(U)-f(V)\|_{2}\le \|f\|_{\lip(\cir)}\|U-V\|_{2}.
\end{align*}
\qed

\begin{rmrk}
	Note that $\|f\|_{\lipp}\le\|f\|_{\lip(\cir)}$ for every Lipschitz function $f:\cir\to\C$. It is worth mentioning that $|\lambda-\mu|\le\pi$ implies $|e^{i\lambda}-e^{i\mu}|\ge\frac{2}{\pi}|\lambda-\mu|$, and this further implies that $\|f\|_{\lip(\cir)}\le\frac{\pi}{2}\|f\|_{\lipp}$. Therefore \eqref{eq27} establishes a better estimate compared to the previously known estimate \eqref{d_p}.
\end{rmrk}
\begin{rmrk}
	It is established in \cite[Theorem~4.2]{Pe09} that the formula~\eqref{doiuniformula}, with $(U, V)$ replaced by pair of contractions $(T,R)$, remains valid provided that $f$ belongs to the class $ \{\, f \in \mathcal{A}(\mathbb{T}) : f' \in \mathcal{A}(\mathbb{T}) \,\}.$
	In this context, however, the measures $E$ and $F$ appearing in Theorem~\ref{Thm6} should be interpreted as the \emph{semi-spectral measures} associated with $T$ and $R$, respectively (see \cite[Section~2.5]{Pe09} for the definition), rather than as spectral measures. For arbitrary Lipschitz functions, such as trigonometric polynomials, it is unknown whether the formula~\eqref{doiuniformula} holds in its present form for contractions.

\end{rmrk}

\section{Open problems}\label{Sec4}
\begin{Problem}
The constant in Theorem \ref{Thm3}, viz., $2^{\frac{1}{p}}\left[\frac{2+(1-\|T_0\|^2)^{1/2}}{(1-\|T_0\|^2)^{1/2}}\right]~d_p $ blows up  as $\|T_0\|$ gets close to $1$. It will be interesting to know if there a function $f$ which is Lipschitz but not in the disk algebra $\mathcal A(\mathbb T)$ and a sequence of contractions $\{T_{0n}\}$ such that $\| T_1 - T_{0n}\|_p$ remains bounded, $\| T_{0n} \| \uparrow 1$ and $\|f(T_1)-f(T_{0n})\|_p \uparrow  \infty$.
\end{Problem}

Moving on to two variables, let $B^{1}_{\infty,1}(\cir^{2})_{+}$ be the analytic subspace of the Besov space $B^{1}_{\infty,1}(\cir^{2})$. For a detailed exposition on Besov spaces and its properties, see \cite[Section 2]{Peller19}. Peller proved the following in \cite{Peller19}.
\begin{thm}
	Let $1\le p<\infty$ and let $(T_{1},R_{1})$ and $(T_{2},R_{2})$ be pairs of commuting contractions such that $T_{1}-T_{2},~R_{1}-R_{2}\in\mathcal{S}_{p}(\hilh)$. Suppose that $f$ is a function of class $B^{1}_{\infty,1}(\cir^{2})_{+}$. Then $\left(f(T_{1},R_{1})-f(T_{2},R_{2})\right)\in\mathcal{S}_{p}(\hilh)$ and
	\begin{align*}
		\left\|f(T_{1},R_{1})-f(T_{2},R_{2})\right\|_{p}\le{\normalfont\text{const}}\,\|f\|_{B^{1}_{\infty,1}(\cir^{2})}~\max\left\{\|T_{1}-T_{2}\|_{p},~\|R_{1}-R_{2}\|_{p}\right\}.
	\end{align*}
\end{thm}

Let $(T_{1},R_{1})$ and $(T_{2},R_{2})$ be two pairs of doubly commuting contractions acting on $\hilh$, i.e., $T_i R_i = R_i T_i$ and $T_i^* R_i = R_i T_i^*$ for $i=1,2$. Let $(U_{1},V_{1})$ and $(U_{2},V_{2})$ be the commuting unitary dilations of the pairs $(T_{1},R_{1})$ and $(T_{2},R_{2})$ respectively on the Hilbert spaces $\hilk_1\supset\hilh$ and $\hilk_2\supset\hilh$. If $P_{\hilh},~Q_{\hilh}$ are the projections from $\hilk_1$ onto $\hilh$ and $\hilk_2$ onto $\hilh$ respectively, then for any trigonometric polynomial $p(z_{1},z_{2})$, we have
\begin{align}
	\label{Opeq2}&p(T_{1},R_{1})=P_{\hilh}\,p(U_{1},V_{1})\big|_{\hilh}\hspace{0.5cm}\text{ and }\hspace{0.5cm}p(T_{2}, R_{2})=Q_{\hilh}\,p(U_{2},V_{2})\big|_{\hilh},
\end{align}
and by a repetition of the argument employed in this paper, one can extend \eqref{Opeq2} to every $f\in C(\cir^{2})$. We conclude with two open problems.

\begin{Problem}
	Let $1<p<\infty$ and let $(T_{1},R_{1})$ and $(T_{2},R_{2})$ be two pairs of doubly commuting contractions on $\hilh$  satisfying the following conditions:
	\begin{enumerate}[{\normalfont(i)}]
		\item $T_2$ and $R_{2}$ are strict contractions.
		\item $T_{1}-T_{2}\in\mathcal{S}_{p}(\hilh),\,R_{1}-R_{2}\in\mathcal{S}_{p}(\hilh)$.
	\end{enumerate}
	Suppose that $f$ is a function of class $\lip(\cir^{2})$. Then is it true that 
	$$\left(f(T_{1},R_{1})-f(T_{2},R_{2})\right)\in\mathcal{S}_{p}(\hilh)?$$ 
	If yes, does there exist a constant $d_{p,T_{2},R_{2}}>0$ such that
	\begin{align*}
		\left\|f(T_{1},R_{1})-f(T_{2},R_{2})\right\|_{p}\le d_{p,T_{2},R_{2}}\|f\|_{\lip(\cir^{2})}~\max\left\{\|T_{1}-T_{2}\|_{p},~\|R_{1}-R_{2}\|_{p}\right\}
	\end{align*}
	holds, where 
	\begin{align*}
		\|f\|_{\lip(\cir^{2})}&=\inf\Big\{K>0:~~|f(z_{1},w_{1})-f(z_{2},w_{2})|\\
		&\le K\max\left\{|z_{1}-z_{2}|,\,|w_{1}-w_{2}|\right\}\Big\}?
	\end{align*}	
\end{Problem}

\subsection*{Acknowledgment}
The authors thank the anonymous referee for a careful reading and pertinent suggestions. T. Bhattacharyya is supported by a J C Bose Fellowship JCB/2021/000041 of SERB. A. Chattopadhyay is supported by the Core Research Grant (CRG), File No: CRG/2023/004826, of SERB. S. Giri acknowledges the support by the Prime Minister's Research Fellowship (PMRF), Government of India. C. Pradhan acknowledges support from JCB/2021/000041, IoE post-doctoral fellowship as well as NBHM post-doctoral fellowship. This research is supported by the DST FIST program-2021 [TPN-700661].

\end{document}